\documentclass[parskip=half,bibliography=totoc]{scrartcl}

\usepackage{comment}
\usepackage[automark]{scrlayer-scrpage}								
\usepackage{amsfonts}												
\usepackage{amsmath}												
\usepackage{mathtools}												
\usepackage{stmaryrd}												
\usepackage{amssymb}												
\usepackage{mathrsfs}												
\usepackage{accents}												
\usepackage[T1]{fontenc}											
\usepackage[utf8]{inputenc}											
\usepackage[top=1in,bottom=1in,left=1.25in,right=1.25in]{geometry}	
\usepackage{enumerate} 												
\usepackage{authblk}												
\usepackage{abstract}												
\usepackage{etoolbox}												
\usepackage[normalem]{ulem}

\usepackage{tikz-cd}												

\usepackage{amsthm}													

\usepackage[colorlinks=true,citecolor=blue]{hyperref}				
\usepackage[noabbrev]{cleveref}										

\newcommand{\R}{\mathbb{R}}

\newtheoremstyle{mythm}
{}
{}
{\slshape}
{}
{\bfseries\sffamily}
{.}
{ }
{}
\newtheoremstyle{mydef}
{}
{}
{}
{}
{\bfseries\sffamily}
{.}
{ }
{}

\theoremstyle{mythm}
\newtheorem{thm}{Theorem}[section]
\newtheorem{prob}[thm]{Problem}
\newtheorem{prop}[thm]{Proposition}
\newtheorem{cor}[thm]{Corollary}

\theoremstyle{mydef}
\newtheorem{mydef}[thm]{Definition}
\newtheorem{rem}[thm]{Remark}

\clearscrheadfoot
\ihead[]{\headmark}
\ohead[]{\pagemark}
\deffootnote[1em]{0em}{1em}{%
	\textsuperscript{\thefootnotemark}%
}
\setfootnoterule{3em}

\apptocmd{\sloppy}{\hbadness 10000\relax}{}{}

\title{Four-dimensional Einstein manifolds with Heisenberg symmetry}
\author{V.\ Cortés}
\author{A.\ Saha}
\affil{\normalsize  Department of Mathematics \\
University of Hamburg\\
Bundesstra\ss e 55, D-20146 Hamburg, Germany\\
vicente.cortes@uni-hamburg.de, arpan.saha@uni-hamburg.de}

\date{}

\begin{document}
\maketitle

\begin{abstract}
We classify Einstein metrics on $\mathbb{R}^4$ invariant under a four-dimensional group of isometries
including a principal action of the Heisenberg group. We consider metrics which are either Ricci-flat or of negative Ricci curvature.
We show that all of the Ricci-flat metrics, including the simplest ones which are hyper-K\"ahler, are incomplete. 
By contrast, those of negative Ricci curvature contain precisely two complete examples: the 
complex hyperbolic metric and a metric of cohomogeneity one known as the one-loop deformed universal hypermultiplet.
	
\end{abstract}

\clearpage

\tableofcontents
\clearpage

\section{Introduction}
It has been recently shown \cite{CST} that all the known homogeneous quaternionic K\"ahler manifolds of negative 
Ricci curvature with exception of the simplest examples, the quaternionic hyperbolic spaces, admit a canonical deformation to 
a complete quaternionic K\"ahler manifold with an isometric action of cohomogeneity one. The deformation is a special case 
of what is known as the one-loop deformation \cite{RSV2006}. 
The simplest example is a deformation of the complex hyperbolic plane known as the one-loop deformed universal hypermultiplet \cite{AMTV}. 
(The completeness requires the deformation parameter to be non-negative \cite[Proposition 4]{ACDM}.) 
Its isometry group is precisely $\mathrm{O}(2)\ltimes \mathrm{H}$, where $H$ is the three-dimensional Heisenberg group \cite{CST}.

In this paper we determine all Einstein metrics  which are invariant under the action of $\mathrm{SO}(2)\ltimes \mathrm{H}$ on 
$\mathbb{R}^4$. The symmetry assumption reduces the problem to the solution of a system of second order ordinary 
differential equations for a pair of functions $a$, $b$, see  (\ref{firstprime:eq})-(\ref{thirdprime:eq}). The corresponding metrics are 
of the form 
\[ g= dt^2+a(t)(dz+xdy-ydx)^2 + b(t)(dx^2+dy^2).\] 

The system admits 
solutions if and only if the Einstein constant $\Lambda$ is non-positive. 

The Ricci-flat solutions include simple solutions of hyper-K\"ahler type (Proposition \ref{HK:prop}) as well as more complicated solutions (Proposition \ref{pr:flat-case}). They are all incomplete. 

The solutions of negative Ricci curvature are described in 
Proposition \ref{prop:K-constraint}. The stationary solutions are isometric to the complex hyperbolic plane
(Proposition \ref{stat:sec}). The one-loop deformed universal hypermultiplet corresponds to a particular solution 
(Proposition \ref{1lp:prop}), interpolating between a stationary solution 
and another fixed point of the flow 
defined by the subsystem 
(\ref{firstprime:eq})-(\ref{secondprime:eq}).  

The main result is that the only complete $\mathrm{SO}(2)\ltimes \mathrm{H}$-invariant Einstein metrics  
on $\mathbb{R}^4$ are the complex hyperbolic metric and its complete one-loop deformation (Theorem \ref{main:thm}).

{\bfseries Acknowledgements}

This work was supported by the German Science Foundation (DFG) under Germany's Excellence Strategy  --  EXC 2121 ``Quantum Universe'' -- 390833306. 
We thank Maciej Dunajski, \'Angel Murcia and Simon Salamon for useful comments.

\clearpage

\section{Riemannian metrics with Heisenberg symmetry}
In this section we describe the class of metrics for which we will study the Einstein equation. 
\subsection{The Heisenberg group}
Recall that the Heisenberg group $H$ is the unique simply connected nilpotent Lie group of dimension $3$, up to isomorphism. 
We choose to realize it as $\mathbb{R}^3$ endowed with the following product:
\begin{equation} \label{heis:eq}(x,y,z) \cdot (a,b,c) = (a+x,b+y,c+z+ya-xb) \end{equation}
The advantage over other natural realizations (e.g.\ as the group of unipotent upper triangular matrices of rank $3$) is that 
the group $\mathrm{SL}(2,\mathbb{R})$ of unimodular transformations in the $(x,y)$-plane acts by automorphisms in these coordinates. 
This follows from Proposition \ref{aut:prop} below. 
Abbreviating $v=(x,y)^\top$, we can write (\ref{heis:eq}) more compactly as $(v_1,z_1)\cdot (v_2,z_2) = (v_1+v_2,z_1+z_2 -\omega (v_1,v_2))$, 
where $\omega = dx\wedge dy$. This implies the following. 
\begin{prop} \label{aut:prop} For any $A\in \mathrm{GL}(2,\mathbb{R})$ the transformation $(v,z)\mapsto (Av,z\det A )$ is an automorphism of
$H$.
\end{prop}

From (\ref{heis:eq}) we can 
immediately read off that the left-invariant parallelization extending the standard basis at the neutral element is given by
\begin{equation}\label{frame:eq} e_1=\partial_x +y\partial_z,\quad e_2= \partial_y-x\partial_z,\quad e_3=\partial_z,\end{equation}
with non-trivial structure constants determined by $[e_1,e_2] = -2e_3$ or, equivalently, by $de^3=2e^1\wedge e^2$ in terms
of the dual frame $(e^i)=(e_i^*)$.

\begin{prop} 
The isometry group $\mathrm{Isom}(H,g)$ of any left-invariant metric $g$ on $H$ is conjugate to $\mathrm{O}(2)\ltimes  H$ 
in  $\mathrm{Aut}(H)\ltimes  H$.
\end{prop}

\begin{proof}
By \cite{W}, the isometry group of $(H,g)$ is $\mathrm{Aut}(H,g)\ltimes  H$, 
where $\mathrm{Aut}(H,g)=\mathrm{Aut}(H)\cap \mathrm{Isom}(H,g)$. Since 
every isometric automorphism preserves the center and its orthogonal complement,
we see that, up to conjugation in $\mathrm{Aut}(H)$, we have the inclusion 
$\mathrm{Aut}(H,g)\subset \mathrm{O}(2)$. On the other hand, any orthogonal 
transformation of the orthogonal complement of the center extends uniquely to an isometric automorphism.  
This shows that, up to conjugation,  $\mathrm{Isom}(H,g)=
\mathrm{O}(2)\ltimes  H$. 
\end{proof}

\subsection{Principal action of the Heisenberg group on $\mathbb{R}^4$}
Any complete Riemannian metric $g$ on $\mathbb{R}^4$ invariant under 
a principal action of the Heisenberg group $H$ can be brought to the form 
\begin{equation}\label{H-inv:eq}g = dt^2 + g_t, \end{equation}
where $\mathbb{R}^4$ is identified with $\mathbb{R}\times H$ by an $H$-equivariant diffeomorphism and 
$g_t$ is a family of left-invariant metrics on $H$. This form is obtained by identifying the $H$-orbits by means of the normal geodesic flow, 
where $t$ corresponds to the arc length parameter along a normal geodesic.

The action of $\mathrm{Aut}(H)\ltimes  H$ on $H$ trivially extends to $\mathbb{R}^4=\mathbb{R}\times H =\{(t,x,y,z)\}$.

\begin{prop} \label{SO2H:prop}An $H$-invariant Riemannian metric $g= dt^2 + g_t$ on $\mathbb{R}^4=\mathbb{R}\times H$ is invariant under $\mathrm{SO}(2)\subset \mathrm{Aut}(H)$ if and only if 
\begin{equation} \label{gt:eq}g_t =  a(t)(dz+xdy-ydx)^2 + b(t)(dx^2+dy^2),\end{equation}
for some positive smooth functions $a, b\in C^\infty (\mathbb{R})$. 
\end{prop} 
\begin{proof}
 We remark 
that (\ref{frame:eq}) implies 
\begin{equation}e^1=dx,\quad e^2=dy,\quad e^3=dz+xdy-ydx.\end{equation}
Recall that, in terms of the coordinates $(t,x,y,z)$, the group $\mathrm{SO}(2)$ acts simply by rotations in the $(x,y)$-plane. 
The induced action on  $H$-invariant one-forms on $\mathbb{R}^4$ is by rotations in the plane spanned by $e^1, e^2$, 
whereas the one-forms $e^3$ and $dt$ are invariant.  
As a consequence, $g_t$ (and hence $g$) is $\mathrm{SO}(2)$-invariant if and only of it is of the form
(\ref{gt:eq}).\end{proof}
\begin{mydef} $\mathrm{SO}(2)\ltimes H$-invariant metrics on $\mathbb{R}^4$, as described in Proposition \ref{SO2H:prop}, 
will be called metrics with \emph{maximal Heisenberg symmetry}.
\end{mydef}

The main problem studied in this paper is the following. 
\begin{prob}
Determine all Einstein metrics on $\mathbb{R}^4$ with maximal Heisenberg symmetry.\end{prob}
The following consequence of Proposition \ref{SO2H:prop} is used in the calculations 
of the connection and the curvature in the next section. Note also that the map 
$(t,x,y,z) \mapsto (t,y,-x,t)$ is an isometry (in the group $\mathrm{SO}(2)$), 
which can be also used for that purpose. 

\begin{cor} Any metric $g$ with maximal Heisenberg symmetry on $\mathbb{R}^4$ is $\mathrm{O}(2)$-invariant, that is not only 
$\mathrm{SO}(2)$-invariant but, in addition,  invariant under the involution $\sigma : (t,x,y,z)\mapsto (t,y,x,-z)$. The 
surface 
\begin{equation} \Sigma = (\mathbb{R}^4)^\sigma = \{ p\in \mathbb{R}^4 \mid \sigma (p) =p\} = \{ (t,x,x,0)\mid t,x \in \mathbb{R}\}\end{equation} 
is totally geodesic and induces an $H$-invariant foliation of $\mathbb{R}^4$ by totally geodesic surfaces. 
The leaf through a point $p_0=(t_0,x_0,y_0,z_0)$ is given by 
\begin{equation}\Sigma_{p_0}= (x_0,y_0,z_0)\cdot \Sigma = \{  (t,x,y,z)\mid x-y=x_0-y_0,\quad z = (y_0-x_0)(x-x_0)+z_0\}.\end{equation}
\end{cor}
\section{Einstein equation for metrics with maximal Heisenberg symmetry}
In this section we determine the system of ordinary differential equations 
satisfied by Einstein metrics with maximal Heisenberg symmetry. 
First we compute the Levi-Civita connection and Ricci curvature of such metrics.  

Throughout this section 
\begin{equation}\label{metric:eq}g=dt^2 + a(t)(dz+xdy-ydx)^2 + b(t)(dx^2+dy^2)\end{equation}
denotes a metric with maximal Heisenberg symmetry on $\mathbb{R}^4$.  
\subsection{Connection and Ricci curvature}
\begin{prop} The Levi-Civita connection $\nabla$ of a metric (\ref{metric:eq})  with maximal Heisenberg symmetry is given by 
\begin{align*} 
\nabla_{\partial_t}\partial_t&=0,\quad \nabla_{\partial_t}\partial_z= \frac12 (\ln a)'\partial_z,\quad 
\nabla_{\partial_z}\partial_z = -\frac12 a'\partial_t,\\
\nabla_{\partial_t}\partial_x &= \frac{y}{2} \left(\ln \frac{b}{a}\right)'\partial_z +\frac12 (\ln b)'\partial_x,\quad 
\nabla_{\partial_t}\partial_y=\frac{x}{2} \left(\ln \frac{a}{b}\right)'\partial_z +\frac12 (\ln b)'\partial_y,\\
\nabla_{\partial_z}\partial_x &= \frac12 a'y\partial_t-\frac{a}{b}x\partial_z+\frac{a}{b}\partial_y,\quad \nabla_{\partial_z}\partial_y=-\frac12 a'x\partial_t-\frac{a}{b}y\partial_z-\frac{a}{b}\partial_x,\\
\nabla_{\partial_x}\partial_x&= -\frac12 (a'y^2+b')\partial_t +2 \frac{a}{b}xy\partial_z-2\frac{a}{b}y\partial_y,\\
\nabla_{\partial_y}\partial_y&= -\frac12 (a'x^2+b')\partial_t -2 \frac{a}{b}xy\partial_z-2\frac{a}{b}x\partial_x,\\
\nabla_{\partial_x}\partial_y &= \frac12 a'xy\partial_t +\frac{a}{b}(y^2-x^2)\partial_z +\frac{a}{b}y\partial_x +
\frac{a}{b}x\partial_y.
\end{align*}
\end{prop}
\begin{prop} The Ricci curvature $\mathrm{Ric}_g = \sum R_{ij}dx^idx^j$, of $g$ is given in the coordinates $(x^0,x^1,x^2,x^3) = (t,x,y,z)$ by
\begin{align*}
R_{00} &= -\left( \frac{1}{4}((\ln a)')^2 +\frac12((\ln b)')^2 +\frac12(\ln a)'' + (\ln b)''\right) g_{00},\\
R_{11} &= -\frac12 (a''y^2+b'') -\frac14 (\ln a)'b' + 2 \frac{a^2}{b^2}y^2  +\frac{y^2}{4}(\ln a)'a' -\frac12 a'y^2(\ln b)'-2\frac{a}{b},\\
R_{22} &= -\frac12 (a''x^2+b'') -\frac14(\ln a)'b' + 2\frac{a^2}{b^2}x^2+\frac{x^2}{4} (\ln a)'a'  -\frac12 a'x^2(\ln b)' 
-2\frac{a}{b},\\
R_{33}&= \left( -\frac{a''}{2a} +\frac{(a')^2}{4a^2} - \frac{a'b'}{2ab}+2\frac{a}{b^2}\right)g_{33},\\
R_{12} &= \frac12 a''xy -\frac{(a')^2}{4a}xy -2\left(\frac{a}{b}\right)^2xy +\frac12 a'(\ln b)'xy,\\
R_{13} &= \frac12 a''y +\frac{y}{4}a'\left( \ln \frac{b}{a}\right)' -2 \left( \frac{a}{b}\right)^2y +\frac14 a'(\ln b)'y,\\
R_{23} &= -\left( \frac12 a''x +\frac{x}{4}a'\left( \ln \frac{b}{a}\right)' -2 \left( \frac{a}{b}\right)^2x +\frac14 a'(\ln b)'x\right),\\
R_{01} &= 0 =g_{01},\quad R_{02} =0=g_{02},\quad R_{03} =0=g_{03}. 
\end{align*}
\end{prop}
\subsection{Einstein equation}
\begin{cor} A metric (\ref{metric:eq})  with maximal Heisenberg symmetry is Einstein, $\mathrm{Ric}_g = \Lambda g$, with constant $\Lambda$ if and only if 
\begin{equation*}
\begin{split}
 &\frac{1}{4}((\ln a)')^2 +\frac12((\ln b)')^2 +\frac12(\ln a)'' + (\ln b)'' = \frac{a''}{2a} -\frac{(a')^2}{4a^2} + \frac{a'b'}{2ab}-\frac{2a}{b^2} =-\Lambda,\\
 &\frac12 \frac{b''}{b} +\frac14 (\ln a)'(\ln b)'    +\frac{2a}{b^2} =
 -\Lambda .
\end{split}
\end{equation*}
\end{cor}
\begin{cor} The metric $g$ is Einstein with constant $\Lambda$ if and only if 
the functions $\lambda = (\ln a)'$ and $\mu = (\ln b)'$ satisfy the following overdetermined system of 
ordinary differential equations: 
\begin{align}
 \label{first:eq}2\lambda' + 4\mu' +\lambda^2 +2\mu^2  +4\Lambda &= 0,\\
 \label{second:eq}2\lambda'+\lambda^2+2\lambda\mu -\frac{8a}{b^2}+4\Lambda &= 0,\\
 \label{third:eq}2\mu'+2\mu^2+\lambda\mu +\frac{8a}{b^2}+4\Lambda &= 0.
 \end{align}
 The system is equivalent to 
\begin{align}
\label{firstprime:eq}2\lambda'&= -(\lambda^2+2\mu^2+6\lambda\mu +12 \Lambda ) ,\\
\label{secondprime:eq}2\mu' &= 3\lambda \mu +4 \Lambda, \\
\label{thirdprime:eq}0 &= \mu^2 +2\lambda\mu +\frac{4a}{b^2}+4 \Lambda.
 \end{align}
\end{cor}
\begin{proof} 
The first system is obtained by substitution of the variables. 
Adding the equations (\ref{second:eq}) and (\ref{third:eq}) we obtain
\[
2\lambda' +2\mu' +\lambda^2+2\mu^2+3\lambda\mu +8\Lambda =0.
\]
Using this equation we eliminate respectively $\mu'$ and $\lambda'$ from (\ref{first:eq}) 
arriving at (\ref{firstprime:eq}) and (\ref{secondprime:eq}). 
Finally, comparing (\ref{third:eq}) with (\ref{secondprime:eq}) yields (\ref{thirdprime:eq}). 
\end{proof}
\section{Solutions}
\subsection{Classification of stationary solutions}
\label{stat:sec}
We call a solution $(a(t),b(t))$ of the ode system (\ref{firstprime:eq})-(\ref{thirdprime:eq}) \emph{stationary} 
if $\lambda'=\mu'=0$. 
\begin{prop} \label{stat:prop}The stationary solutions $(a,b)$ of the Einstein equations (\ref{firstprime:eq})-(\ref{thirdprime:eq})
are given by 
\begin{equation} \label{stat:eq}a = -\frac{\Lambda}{6} b^2,\quad b= Ce^{\mu t},\end{equation}
where $\mu \neq 0$ and $C>0$ are constants. 
The corresponding Einstein manifold $(\mathbb{R}^4,g)$ is isometric to the complex hyperbolic plane 
of Einstein constant $\Lambda = -\frac32 \mu^2<0$.
\end{prop}
\begin{proof} Since $\lambda$ and $\mu$ are constant for stationary solutions, we see from 
(\ref{thirdprime:eq}) that the function $a/b^2$ is constant and, hence, 
\[ \lambda -2\mu =0.\]
Inserting this into (\ref{firstprime:eq})-(\ref{secondprime:eq}) we obtain
\[ \Lambda = -\frac32 \mu^2\]
and (\ref{thirdprime:eq}) then yields 
\[ \Lambda = -\frac{6a}{b^2}.\]
This shows that $\Lambda <0$ and, hence, $\mu\neq0$.
The above metrics are all homothetic to 
\begin{equation}\label{cxhyp:eq}dt^2+ e^{4t}(dz+xdy-ydx)^2 + e^{2t}(dx^2+dy^2),\end{equation}
by Remark \ref{homoth:rem} below. This is the complex hyperbolic metric of holomorphic sectional curvature $-4$ (i.e.\ $\Lambda =-6$) written 
as a left-invariant metric on the simply transitive solvable Iwasawa subgroup of its group of holomorphic isometries $\mathrm{PSU}(1,2)$. 
\end{proof}
\begin{rem}\label{homoth:rem} 
The two parameters of the solution (\ref{stat:eq}) correspond to the freedom to re-parametrize 
the $t$-variable by an affine transformation. In fact, a transformation of the coordinates 
$(t,x,y,z)$ by a pure translation in $t$ yields another  stationary 
solution but with another $C$-parameter, 
whereas rescaling of the $t$-variable in the coordinate system yields an  
Einstein metric which, up to a constant conformal factor, is a stationary solution in our class (\ref{metric:eq}), 
the latter with another $\mu$-parameter. 

Note that the transformation $(a,b)\mapsto (\tilde{a},\tilde{b})$, where  
\begin{equation} \label{rescale:eq} \tilde{a}(t):= \frac{a(kt)}{k^2},\quad  \tilde{b}(t):= \frac{b(kt)}{k^2},\quad 
 k\in \mathbb{R}\setminus \{0 \}, 
\end{equation}
maps arbitrary solutions of (\ref{firstprime:eq})-(\ref{thirdprime:eq}) to (homothetic) solutions. In this 
way one can always normalize a given solution with $\Lambda < 0$ such that $\Lambda=-6$. 
\end{rem}
\subsection{Ricci-flat solutions}
\label{Ricci-flat:sec}
\begin{prop} \label{HK:prop}There exist solutions $(a,b)$ of the Einstein equations (\ref{firstprime:eq})-(\ref{thirdprime:eq})
with $\lambda = \frac{\ell}{t}$ and $\mu=\frac{m}{t}$, $\ell , m\in \mathbb{R}$. They are all hyper-K\"ahler and of the form 
\begin{equation}\label{eq:ab-hk}
a(t)=a_1|t|^{-2/3},\quad b(t) = b_1|t|^{2/3}, 
\end{equation}
where $a_1, b_1$ are positive constants such that $a_1/b_1^2=1/9$. The maximal domains  
of definition of these (incomplete) metrics are $\mathbb{R}_{>0}\times \mathbb{R}^3$ and 
$\mathbb{R}_{<0}\times \mathbb{R}^3$.  
\end{prop}
\begin{proof}Under the ansatz $\lambda = \frac{\ell}{t}$ and $\mu=\frac{m}{t}$, the equation (\ref{secondprime:eq}) implies
that $\Lambda=0$. The equations (\ref{firstprime:eq})-(\ref{secondprime:eq}) can then 
be easily solved in terms of $(\ell , m)$. We find that $(\ell ,m)$ is one of the 
following: $(-\frac23,\frac23), (-\frac23,\frac43), (2,0), (0,0)$. However, the last three cases 
are clearly inconsistent with equation (\ref{thirdprime:eq}). (The case $(0,0)$ is also 
excluded, because such a solution would be stationary contrary to Proposition \ref{stat:prop}.) 
So we are left with studying equation (\ref{thirdprime:eq}) in the case $(\ell ,m)=(-\frac23,\frac23)$. 
Inserting $a=a_1|t|^\ell =a_1|t|^{-2/3}$ and $b=b_1|t|^m=b_1|t|^{2/3}$ we obtain $a_1/b_1^2=1/9$.

The given metric can now be explicitly shown to be hyper-K\"ahler. Consider the following two-forms:
\begin{align*}
	\omega_1 &= \sqrt{a}\,dt \wedge (dz + xdy - ydx) + \mathrm{sign}(t) bdx \wedge dy\\
	&=\sqrt{a_1}|t|^{-1/3}dt \wedge (dz + xdy - ydx) + 3\,\mathrm{sign}(t) \sqrt{a_1}|t|^{2/3}dx \wedge dy,\\
	\omega_2 &=   \sqrt{b}\,dt \wedge dy + \mathrm{sign}(t)\sqrt{ab}(dz + xdy - ydx)\wedge dx\\
	&=\sqrt{b_1}|t|^{1/3}dt \wedge dy + \mathrm{sign}(t)\sqrt{a_1b_1}(dz + xdy - ydx)\wedge dx,\\
	\omega_3 &= \sqrt{b}\,dt \wedge dx + \mathrm{sign}(t) \sqrt{ab}\,dy\wedge(dz + xdy - ydx)\\
	&=\sqrt{b_1}|t|^{1/3}dt \wedge dx +\mathrm{sign}(t) \sqrt{a_1b_1}\,dy\wedge(dz + xdy - ydx).
\end{align*}
These are (anti-)self-dual and closed, and so form a hyper-K\"ahler structure.
\end{proof}
\begin{rem} The incomplete hyper-K\"ahler metrics described in Proposition \ref{HK:prop} 
are all homothetic to a single metric, compare (\ref{rescale:eq}). The metric can be obtained from a Gibbons-Hawking ansatz and admits a conformal rescaling to 
a complete left-invariant metric on the solvable Iwasawa subgroup 
of $\mathrm{SU}(1,2)$ \cite[Section 3.2.2]{DH}. The metric does also appear in the study of collapsing hyper-K\"ahler metrics on 
K3 surfaces \cite{HSVZ}, as we learned from Simon Salamon \cite{S}.
\end{rem}
\begin{prop}\label{pr:flat-case}
	If $(a,b)$ is a Ricci-flat solution of the Einstein equations (\ref{firstprime:eq})-(\ref{thirdprime:eq}) not isometric to \eqref{eq:ab-hk}, then the associated functions $\lambda$ and $\mu$ satisfy
	\begin{equation}\label{eq:flat-gen-soln}
	\frac{2}{3\mu}\left(\mp _{2}F_1\bigg(\!\!-\!\frac{3}{4},\frac{1}{2};\frac{1}{4};-C|\mu|^{4/3}\bigg)+1\right) = t - t_0, \qquad \lambda =  \pm \frac{\mu\sqrt{1+C|\mu|^{4/3}}}{1\mp \sqrt{1+C|\mu|^{4/3}}},
	\end{equation}
	where $C >0$ and $t_0$ are constants and $_2F_1$ is the hypergeometric function. The corresponding metrics are incomplete. In fact, the maximal domains of definition of these metrics are 
	\begin{equation}
	\begin{split}
	\left]-\infty, t_0 - \frac{2C^{3/4}}{3\sqrt{\pi}}\,\Gamma\bigg(\frac{1}{4}\bigg)\Gamma\bigg(\frac{5}{4}\bigg)\right[ \times \R^3, \quad &\left]t_0 - \frac{2C^{3/4}}{3\sqrt{\pi}}\,\Gamma\bigg(\frac{1}{4}\bigg)\Gamma\bigg(\frac{5}{4}\bigg), t_0\right[ \times \R^3, \\
	\left]t_0 + \frac{2C^{3/4}}{3\sqrt{\pi}}\,\Gamma\bigg(\frac{1}{4}\bigg)\Gamma\bigg(\frac{5}{4}\bigg), +\infty\right[ \times \R^3, \quad &\left]t_0, t_0 + \frac{2C^{3/4}}{3\sqrt{\pi}}\,\Gamma\bigg(\frac{1}{4}\bigg)\Gamma\bigg(\frac{5}{4}\bigg)\right[ \times \R^3.
	\end{split}
	\end{equation}
\end{prop}

\begin{proof}
	Setting $\Lambda=0$ in  (\ref{firstprime:eq})-(\ref{thirdprime:eq}) gives us
	\begin{align}
	2\lambda'&=-(\lambda^2 + 6\lambda\mu+ 2\mu^2),\label{eq:dlamb}\\
	2\mu'&=3\lambda\mu,\label{eq:dmu}\\
	-\frac{4a}{b^2}&=\mu(\mu + 2\lambda).\label{eq:constraint}
	\end{align}
	On a domain where $\mu$ and $\mu + 2\lambda$ are non-vanishing, \eqref{eq:dlamb} and \eqref{eq:dmu} imply \begin{equation*}
	\frac{d(\mu(\mu  + 2\lambda ))}{\mu(\mu  + 2\lambda )} = \lambda dt - 2\mu dt.
	\end{equation*}
	Integrating  and then exponentiating both sides gives
	\begin{equation*}
	\mu^2 + 2\lambda\mu = -\frac{4ka}{b^2},
	\end{equation*}
	where $k$ is a non-zero constant of integration. Notice however that given $\lambda = (\ln a)'$ and $\mu= (\ln b)'$, the positive  functions  $a$ and $b$ are determined only up to overall positive constant factors. Thus, the constant $k$ may be absorbed into this indeterminacy so that the constraint \eqref{eq:constraint} is satisfied, provided that $k>0$.
	
	In particular, this argument fails when either $\mu$ or $\mu +2\lambda$ vanish. In fact, \eqref{eq:constraint} then necessarily means that $a$ vanishes, which is not allowed. Therefore, the constraint amounts to stipulating that $\mu$ and $\mu +2\lambda$ are non-vanishing on the domain of definition and of opposite sign.
	
	We will now describe general non-stationary solutions of \eqref{eq:dlamb} and \eqref{eq:dmu}. If $\lambda$ is a constant function, then so is $\mu$. We may thus assume that $\lambda$ is not everywhere $0$. As $\mu$ is constrained to be non-vanishing, $\mu'=\frac{3}{2}\lambda\mu$ must also be non-vanishing on the (open) complement of the vanishing set of $\lambda$. On this open set, we may regard $t$, and hence $\lambda(t)$, as an implicit function of $\mu$. Then $\lambda$ satisfies the following ode:
	\begin{equation*}
		\frac{d\lambda}{d\mu} = \frac{\lambda'}{\mu'}=-\frac{\lambda^2 + 6\lambda\mu+ 2\mu^2}{3\lambda\mu}.
	\end{equation*}
	Define a function $\nu$ by $\lambda = \mu\nu$. Substituting this into the above equation and rearranging the terms gives us
	\begin{equation*}
	\mu\frac{d\nu}{d\mu}= -\frac{4\nu^2 +  6\nu +2}{3\nu}.
	\end{equation*}
	There are two cases to be considered now: either the numerator of the right-hand side is identically zero or it is not. 
	
	Let us suppose the first case, that is
	\begin{equation*}
	4\nu^2 +  6\nu +2 = 2(\nu + 1)(2\nu +1) = 0.
	\end{equation*}
	Then $\nu$ takes the value $-1$ or $-\frac{1}{2}$. Note that if $\nu = -\frac{1}{2}$, then
	\begin{equation*}
	\mu + 2\lambda = \mu +2\mu\nu = 0.
	\end{equation*}
	So this is not allowed, and $\nu$ must necessarily be $-1$. Thus, $\lambda = -\mu$ and $\mu'=\frac{3}{2}\lambda\mu=-\frac{3}{2}\mu^2$. Up to constant shifts in $t$, this gives the same solution as in \eqref{eq:ab-hk}, and so is excluded as well.
	
	So $4\nu^2 +  6\nu +2$ cannot be identically zero. On the complement of its vanishing set, we we may separate the variables and integrate to obtain
	\begin{equation*}
	\ln|\nu + 1| - \frac{1}{2}\ln|2\nu + 1| = -\frac{2}{3}\ln|\mu| + \mathrm{const.}
	\end{equation*}
	Multiplying by $-2$ throughout and then exponentiating  both sides gives us
	\begin{equation}\label{eq:C}
	\frac{2\nu + 1}{(\nu + 1)^2} = -C|\mu|^{4/3},
	\end{equation}
	where $C$ is some non-zero constant.  
Then solving for $\nu$, we get
	\begin{equation*}
	\begin{split}
	\nu &= \pm \frac{\sqrt{1+C|\mu|^{4/3}}}{1\mp \sqrt{1+C|\mu|^{4/3}}}.
	\end{split}
	\end{equation*}
	Thus, $\lambda$ as a function of $\mu$ is given by
	\begin{equation*}
	\lambda = \mu\nu = \pm \frac{\mu\sqrt{1+C|\mu|^{4/3}}}{1\mp \sqrt{1+C|\mu|^{4/3}}}.
	\end{equation*}
	To now obtain $\mu$ as function of $t$, we substitute the above expression into \eqref{eq:dmu}:
	\begin{equation*}
	2\mu'=\pm\frac{3\mu^2\sqrt{1+C|\mu|^{4/3}}}{1\mp \sqrt{1+C|\mu|^{4/3}}}.
	\end{equation*}
	Separating the variables and integrating gives us the equation 
	\begin{equation}\label{hypergeom:eq}
		\frac{2}{3\mu}\left(\mp _2F_1\bigg(\!\!-\!\frac{3}{4},\frac{1}{2};\frac{1}{4};-C|\mu|^{4/3}\bigg)+1\right) = t - t_0,
	\end{equation}
	where $t_0$ is a constant of integration. To see this we remark that the hypergeometric function $_2F_1( a, b;c;x)$ 
	for $c=a+1$, $a\neq 0$, is related to the incomplete beta function $B_x(a,1-b)$ by 
	\[ B_x(a,1-b) = {_2F_1( a, b;a+1;x) \frac{x^a}{a}}.\]
	This implies that $_2F_1( a, b;a+1;x)$ satisfies the first order ode $F'(x)= \frac{a((1-x)^{-b}-F(x))}{x}$, which leads to 
	(\ref{hypergeom:eq}).

	To determine the maximal domains of definition of the metric, we determine the values of $t$ for which either at least one of $\lambda$ and $\mu$ becomes infinite or for which we have $\mu(\mu + 2\lambda) = 0$. 
	
	We find that for the upper branch of the solution, that is
	\begin{equation*}
	\frac{2}{3\mu}\left(- _{2}F_1\bigg(\!\!-\!\frac{3}{4},\frac{1}{2};\frac{1}{4};-C|\mu|^{4/3}\bigg)+1\right) = t - t_0, \qquad \lambda =   \frac{\mu\sqrt{1+C|\mu|^{4/3}}} {1- \sqrt{1+C|\mu|^{4/3}}},
	\end{equation*}
	taking the limit $\mu\rightarrow 0^\pm$ gives us  $t = t_0$ and $\lambda \rightarrow \mp \infty$. By contrast, on the lower branch of the solution, that is
	\begin{equation*}
	\frac{2}{3\mu}\left(+ _{2}F_1\bigg(\!\!-\!\frac{3}{4},\frac{1}{2};\frac{1}{4};-C|\mu|^{4/3}\bigg)+1\right) = t - t_0, \qquad \lambda =   -\frac{\mu\sqrt{1+C|\mu|^{4/3}}}{1+ \sqrt{1+C|\mu|^{4/3}}},
	\end{equation*}
	the limit $\mu \rightarrow 0^\pm$ gives $t\rightarrow \pm\infty$ and $\lambda \rightarrow 0$. 
	
	Meanwhile, setting $\mu + 2\lambda=0$ is the same as setting $\nu = \frac{\lambda}{\mu} = -\frac12$, giving us:
	\begin{equation*}
		\begin{split}
			-\frac{1}{2} &= \pm \frac{\sqrt{1+C|\mu|^{4/3}}}{1\mp \sqrt{1+C|\mu|^{4/3}}}.
	\end{split}
	\end{equation*}
This is solved only by $\mu = 0$ on the lower branch of the solution, and therefore for no finite value of $t$.

In the case that $C$ is positive, we can take the limit $\mu \rightarrow \pm \infty$ to obtain on the upper branch
\begin{equation*}
	t = t_0 \mp \frac{2C^{3/4}}{3\sqrt{\pi}}\,\Gamma\bigg(\frac{1}{4}\bigg)\Gamma\bigg(\frac{5}{4}\bigg), \quad \lambda \rightarrow \mp \infty,
\end{equation*}
and on the lower branch
\begin{equation*}
t = t_0 \pm \frac{2C^{3/4}}{3\sqrt{\pi}}\,\Gamma\bigg(\frac{1}{4}\bigg)\Gamma\bigg(\frac{5}{4}\bigg), \quad \lambda \rightarrow \mp \infty.
\end{equation*}
Note that the above cases automatically take care of the limits in which $\lambda$ becomes infinite. The above limits are obtained by specializing the 
asymptotics for $|x|\rightarrow \infty$ of the hypergeometric function $F(x)= {_{2}F_1(a,b;c;x)}$ for $a-b\not\in\mathbb{Z}$ to $(a,b,c)= (-\frac{3}{4},\frac{1}{2};\frac{1}{4})$:
\[ F(x) \sim \frac{\Gamma (b-a)\Gamma(c)}{\Gamma(b)\Gamma (c-a)}(-x)^{-a} +\frac{\Gamma (a-b)\Gamma(c)}{\Gamma(a)\Gamma (-b)}(-x)^{-b}.
\]

Putting everything together, we find that the maximal domains of definition for $t$ are  the open intervals $\left]-\infty, t_0\right[$ and $]t_0,+\infty[$ when $C < 0$, and the following open intervals when $C>0$:
\begin{equation*}
\begin{split}
	\left]-\infty, t_0 - \frac{2C^{3/4}}{3\sqrt{\pi}}\,\Gamma\bigg(\frac{1}{4}\bigg)\Gamma\bigg(\frac{5}{4}\bigg)\right[, \quad \left]t_0 - \frac{2C^{3/4}}{3\sqrt{\pi}}\,\Gamma\bigg(\frac{1}{4}\bigg)\Gamma\bigg(\frac{5}{4}\bigg), t_0\right[, \\
	\left]t_0 + \frac{2C^{3/4}}{3\sqrt{\pi}}\,\Gamma\bigg(\frac{1}{4}\bigg)\Gamma\bigg(\frac{5}{4}\bigg), +\infty\right[,\quad \left]t_0, t_0 + \frac{2C^{3/4}}{3\sqrt{\pi}}\,\Gamma\bigg(\frac{1}{4}\bigg)\Gamma\bigg(\frac{5}{4}\bigg)\right[. 
\end{split}
\end{equation*}
Now that we have described all the solutions to the ode system, we check which of them satisfy the sign constraint $\mu (\mu + 2 \lambda)<0$  to determine which of them correspond to Riemanninan metrics. Dividing the sign constraint by $\mu^2>0$, we find that it is equivalent to $2\nu + 1 < 0$. From \eqref{eq:C} we see that this happens precisely when $C>0$.
\end{proof}
\begin{rem}By taking $t$ purely imaginary and the integration constant $t_0$ complex, one can similarly describe 
Ricci-flat Lorentzian metrics of the form 
\[ g=-dt^2 + a(t)(dz+xdy-ydx)^2 + b(t)(dx^2+dy^2)\] 
from solutions of 
(\ref{hypergeom:eq}) with $C<0$. As in the Riemannian case, these are $\mathrm{O}(2)\ltimes H$-invariant. Lorentzian solutions of the 
Einstein equations invariant under a principal action of a three-dimensional Lie group with space-like orbits have been studied as cosmological models in general relativity, see 
\cite{EM}.\end{rem}

\begin{rem}
	The limit $C \rightarrow +\infty$ is in fact well-defined. In this limit, \eqref{eq:flat-gen-soln} becomes
	\begin{equation*}
	\frac{2}{3\mu} = t - t_0, \qquad \lambda =  -\mu.
	\end{equation*}
	A constant shift $t \mapsto t - t_0$ then reproduces the solution \eqref{eq:ab-hk}.
\end{rem}
\subsection{The one-loop deformed universal hypermultiplet}
In this section we exhibit a family of solutions of the Einstein equations (\ref{firstprime:eq})-(\ref{thirdprime:eq}) 
with $\Lambda=-6$ depending on a real parameter $c$. The solution is stationary only for $c=0$, in which case the metric is 
the complex hyperbolic metric (\ref{cxhyp:eq}). 

Let $c$ be a real constant and let $I$ be a connected component of the set 
\begin{equation} \{ \rho \in \mathbb{R} \mid \rho\neq 0,\; \rho+c>0\;\mbox{and}\;\rho +2 c>0\}. \end{equation}
 Let $\rho : J \stackrel{\sim}{\rightarrow}I$, $t\mapsto \rho (t)$, be a (maximal) solution of the differential equation 
\begin{equation} \label{ode:eq}\rho'(t) = 2\rho(t) \sqrt{\frac{\rho(t) +c}{\rho(t)+2c}}\end{equation}
which is defined on some interval $J$ and has the interval $I$ as its range. 
The  functions
\begin{equation}\label{ab:eq}
a(t) = \frac{\rho(t) +c}{4\rho(t)^2(\rho(t) +2c)}\quad\mbox{and}\quad b(t)=\frac{\rho (t)+2c}{2\rho(t)^2}
\end{equation}
are  positive on their domain $J$.

Recall (see Remark \ref {homoth:rem}) that the Einstein constant $\Lambda$ of a solution of (\ref{firstprime:eq})-(\ref{thirdprime:eq}) 
is either zero or the metric can be rescaled such that $\Lambda$ is any constant negative number. 
\begin{prop} \label{1lp:prop} The functions  $a(t)$ and $b(t)$ defined by (\ref{ab:eq}) and (\ref{ode:eq}) constitute a one-parameter family of 
solutions of the Einstein equations (\ref{firstprime:eq})-(\ref{thirdprime:eq}) with 
$\Lambda=-6$. The corresponding metrics are complete if and only if $c>0$ and $I=\{ \rho \mid \rho >0\}$. 
\end{prop}
\begin{proof}
Writing the metric $g=dt^2+a(t)(dz+xdy-ydx)^2 + b(t)(dx^2+dy^2)$ in terms of the 
coordinates $(\rho , x,y,z)$ instead of $(t,x,y,z)$ shows that it coincides with the 
one-loop deformed universal hypermultiplet metric, as given in equation (1.1) of \cite{CS}. (For the physical origins and significance of this metric see \cite{AMTV,RSV2006}.) The metric is not only Einstein of Einstein constant $-6$ but is 
half conformally flat and is complete if and only if $c>0$ and $I=\{ \rho \mid \rho >0\}$, see \cite{ACDM}. 
Moreover, it was shown in \cite[Theorem 4.5]{CST} that for $c\neq 0$ the metric has the isometry group $\mathrm{O}(2) \rtimes H$, where $H$ denotes the Heisenberg group. (For $c=0$ the metric is the complex hyperbolic metric discussed in Section \ref{stat:sec}.) This proves Proposition \ref{1lp:prop}.

Alternatively, one can check directly that the functions $a(t)$ and $b(t)$ solve the system (\ref{firstprime:eq})-(\ref{thirdprime:eq}).
In fact, the equations (\ref{secondprime:eq}) and (\ref{thirdprime:eq}) are easily checked and (\ref{thirdprime:eq})
implies (\ref{firstprime:eq}) on the set where $\mu\neq 0$. The latter is shown by differentiating  (\ref{thirdprime:eq})
and using the simple equation 
\[ \left( \frac{a}{b^2}\right)' = \frac{a}{b^2}(\lambda -2 \mu).\]
A short calculation shows that for the above functions $a$ and $b$, the function $\mu$ 
vanishes only if $c<0$ and $-4c\in I$. In that case, the zero is at $\rho = -4c$, i.e.\ at 
$t=\rho^{-1}(-4c)$. The equation (\ref{firstprime:eq}) follows by continuity, since the complement 
of the zero set is dense. 
\end{proof}

\begin{prop}
	A solution $(\lambda(t),\mu(t))$ of the Einstein equations \eqref{firstprime:eq}-\eqref{thirdprime:eq} corresponding to the one-loop deformed universal hypermultiplet satisfies the following polynomial constraint of degree $4$:
	\begin{equation}\label{eq:P}
		P(\lambda, \mu):=(\lambda + \mu)^3\mu -4(3\lambda^2 +18\lambda\mu +  11 \mu^2) + 512 = 0.
	\end{equation}
\end{prop}
\begin{proof}
	From \eqref{ab:eq}, we obtain the following parametrisation of $\lambda$ and $\mu$ in terms of $\rho$:
	\begin{align}
		\label{lamb-rho:eq}\lambda &= (\ln a)' = \frac{\rho'}{\rho}\left(\frac{\rho}{\rho + c} - 2 - \frac{\rho}{\rho + 2c}\right),\\
	\label{mu-rho:eq}	\mu &= (\ln b)' = \frac{\rho'}{\rho}\left(\frac{\rho}{\rho + 2c} - 2\right)=-\frac{\rho'}{\rho}\left(\frac{\rho +4c}{\rho + 2c}\right).
	\end{align}
In particular, on a domain where $\mu$ is non-vanishing, we have $\frac{4c}{\rho} + 1 \neq 0$ and we can combine the above equations to get 
\begin{equation}\label{eq:lplusm}
	\frac{\lambda + \mu}{\mu} = \left(\frac{\rho}{\rho + c} - 4 \right)\bigg(\frac{\rho}{\rho + 2c} - 2\bigg)^{-1} = 2 + \bigg(\frac{4c}{\rho} + 1\bigg)^{-1}\bigg(\frac{c}{\rho} + 1\bigg)^{-1}. 
\end{equation}
By subtracting $2$ throughout, we see that this implies that $\lambda - \mu$ is non-vanishing. In particular, we have a quadratic equation in $\frac{c}{\rho}$, which we can then solve to obtain
\begin{equation}\label{eq:D}
	\frac{c}{\rho} = \frac{-5 \pm \sqrt{D}}{8} \quad \mbox{where}\quad D = \frac{9\lambda + 7\mu}{\lambda - \mu}.
\end{equation}
Now substituting \eqref{ode:eq} into \eqref{lamb-rho:eq} and \eqref{mu-rho:eq}, we get
\begin{equation}
	\lambda + \mu = 2\sqrt{\frac{\rho + c}{\rho + 2c}}\left(\frac{\rho}{\rho + c} - 4 \right) = -\frac{2(3 + 4c/\rho)}{\sqrt{(1 + 2c/\rho)(1 + c/\rho)}}.
\end{equation}
Upon substituting \eqref{eq:D} into the above and eliminating the square roots, we then obtain
\begin{equation}
	(\lambda + \mu)^2P(\lambda, \mu)=0
\end{equation}
From \eqref{eq:lplusm} we see that $\lambda + \mu$ vanishes if and only if $(3\rho + 4c)(\rho + 2c)$ vanishes. Since this is not generically the case, the constraint \eqref{eq:P} follows. 
\end{proof}

\begin{rem}
	More generally, if we introduce the polynomial
	\begin{equation*}
		P_\Lambda(\lambda, \mu):=(\lambda + \mu)^3\mu +\frac{2\Lambda}{3}(3\lambda^2 +18\lambda\mu + 11 \mu^2) + \frac{128\Lambda^2}{9},
	\end{equation*}
	 then \eqref{firstprime:eq} and \eqref{secondprime:eq} imply
	\begin{equation}\label{eq:Plamb}
		P_\Lambda' = - 3\mu P_\Lambda.
	\end{equation}
	So the vanishing set of $P_\Lambda(\lambda,\mu)$ contains a flowline. When $\Lambda = - 6$, this becomes the constraint \eqref{eq:P}. 
\end{rem}

\begin{prop}\label{prop:K-constraint}
 Any solution $(\lambda(t),\mu(t))$ of the Einstein equations \eqref{firstprime:eq}-\eqref{thirdprime:eq} with negative Einstein constant $\Lambda$ satisfies the following constraint on each maximal domain of definition:
 \begin{equation}\label{eq:flowlines}
 \left(\frac{\mu^2 +\lambda\mu +2\Lambda}{\mu^2 +2\lambda\mu +4\Lambda}\right)^3 P_\Lambda(\lambda,\mu) = K,
 \end{equation}
 for some constant $K$. Furthermore, when $K=0$, the solution has to be one of the following:
 	\begin{itemize}
 		\item Stationary solutions at $\pm(2\sqrt{-2\Lambda/3},\sqrt{-2\Lambda/3})$.
 		\item Solutions with maximal domains of definition $\left]-\infty,t_0\right[$ and $\left]t_0,+\infty\right[$ which are given by
 		\begin{equation}\label{eq:sol1}
 		\mu =  \sqrt{-\frac{2\Lambda}{3}} \frac{ e^{\sqrt{-6\Lambda}(t - t_0)} \pm 1}{ e^{\sqrt{-6\Lambda}(t - t_0)} \mp 1}, \quad \lambda =  2\sqrt{-\frac{2\Lambda}{3}}\frac{ e^{2\sqrt{-6\Lambda}(t - t_0)} \pm 4e^{\sqrt{-6\Lambda}(t - t_0)} + 1}{ e^{2\sqrt{-6\Lambda}(t - t_0)} - 1}.
 		\end{equation}
 		\item A solution with maximal domains of definition $\left]-\infty,t_0\right[$ and $\left]t_0,+\infty\right[$ which is given by
 		\begin{equation}\label{eq:sol2}
 		\begin{split}
 			&\lambda = \frac{\mathrm{sign}(t-t_0)8\sqrt{-2\Lambda/3}}{\sqrt{3 + 18\sigma +11\sigma^2- \sqrt{(1-\sigma)^3(9+7\sigma)}}},\quad \mu = \lambda\sigma, \quad \mbox{where $0<\sigma<1$,} \\
 			&\int_{0}^\sigma \frac{-8\sqrt{3 +18u + 11u^2 -\sqrt{(1-u)^3(9+7u)}}\,du}{(1-u)(1-5u)(9+7u) - 3(1+3u)\sqrt{(1-u)^3(9+7u)}}= \sqrt{-\frac{2\Lambda}{3}}|t-t_0|.
 		\end{split}
 		\end{equation}
 		\item A solution with maximal domains of definition $\left]-\infty,t_0\right[$ and $\left]t_0,+\infty\right[$ which is given by
 		\begin{equation}\label{eq:sol3}
 		\begin{split}
 		&\lambda = \frac{-\mathrm{sign}(t-t_0)8\sqrt{-2\Lambda/3}}{\sqrt{3 + 18\sigma +11\sigma^2+ \sqrt{(1-\sigma)^3(9+7\sigma)}}},\quad \mu = \lambda\sigma, \quad \mbox{where $-1<\sigma<\frac12$,} \\
 		&\int_{-1}^\sigma \frac{8\sqrt{3 +18u + 11u^2 +\sqrt{(1-u)^3(9+7u)}}\,du}{(1-u)(1-5u)(9+7u) + 3(1+3u)\sqrt{(1-u)^3(9+7u)}}= \sqrt{-\frac{2\Lambda}{3}}|t-t_0|.
 		\end{split}
 		\end{equation}
 		\item Solutions defined for all $t\in \mathbb{R}$ given by
 		\begin{equation}\label{eq:sol4}
 		\begin{split}
 		&\lambda = \frac{\pm 8\sqrt{-2\Lambda/3}}{\sqrt{3 + 18\sigma +11\sigma^2+ \sqrt{(1-\sigma)^3(9+7\sigma)}}},\quad \mu = \lambda\sigma,\quad \mbox{where $\frac12<\sigma<1$,} \\
 		&\int_{3/4}^\sigma \frac{\mp 8\sqrt{3 +18u + 11u^2 +\sqrt{(1-u)^3(9+7u)}}\,du}{(1-u)(1-5u)(9+7u) + 3(1+3u)\sqrt{(1-u)^3(9+7u)}}= \sqrt{-\frac{2\Lambda}{3}}(t-t_0).
 		\end{split}
 		\end{equation}
 	\end{itemize}
 	With the exception of \eqref{eq:sol2}, all of them correspond to Riemannian metrics.
\end{prop}

\begin{proof}
	Equations \eqref{firstprime:eq} and \eqref{secondprime:eq} imply 
	\begin{align}
		(\mu^2 +2\lambda\mu +4\Lambda)' &= (\lambda - 2\mu)(\mu^2 +2\lambda\mu +4\Lambda),\label{eq:flow1}\\
		(\mu^2 +\lambda\mu +2\Lambda)' &= (\lambda -\mu)(\mu^2 +\lambda\mu +2\Lambda).\label{eq:flow2}
	\end{align}
	On the domain where $\mu^2 + 2\lambda \mu + 4\Lambda$ is non-vanishing, \eqref{eq:flow1} can be written as
	\begin{equation*}
		\frac{d(\mu^2 +2\lambda\mu +4\Lambda)}{\mu^2 +2\lambda\mu +4\Lambda} = \lambda dt - 2\mu dt.
	\end{equation*}
	Integrating and then exponentiating gives
	\begin{equation*}
		\mu^2 +2\lambda\mu +4\Lambda = -\frac{4ka}{b^2},
	\end{equation*}
	where $k$ is a non-zero constant of integration. As we had noted in the proof of Proposition~\ref{pr:flat-case}, given $\lambda = (\ln a)'$ and $\mu= (\ln b)'$, the functions $a$ and $b$ are determined only up to overall non-zero constant factors. Thus, the constant $k$ may be absorbed into this indeterminacy so that \eqref{thirdprime:eq} is automatically 
	satisfied, provided that $k>0$. So, we see that \eqref{thirdprime:eq} is equivalent to the condition $\mu^2 +2\lambda\mu +4\Lambda<0$.
	
	Given that $\mu^2 +2\lambda\mu +4\Lambda$ is nowhere vanishing, \eqref{eq:Plamb}, \eqref{eq:flow1}, and \eqref{eq:flow2} can be combined into a single equation:
	\begin{equation*}
		\left(\frac{(\mu^2 +\lambda\mu +2\Lambda)^\ell P_\Lambda^m}{(\mu^2 +2\lambda\mu +4\Lambda)^n}\right)'= (\ell(\lambda - \mu) -3m\mu - n(\lambda - 2\mu))\frac{(\mu^2 +\lambda\mu +2\Lambda)^\ell P_\Lambda^m}{(\mu^2 +2\lambda\mu +4\Lambda)^n},
	\end{equation*}
	where $\ell, m, n$ are arbitrary non-negative integers. In particular, we see that the right-hand side vanishes for the choice $\ell = 3, m = 1, n = 3$. The constraint \eqref{eq:flowlines} follows.
	
	For $K=0$, we have either
	\begin{equation*}
		\mu^2 + \lambda \mu + 2\Lambda = 0 \quad \mbox{or}\quad P_\Lambda(\lambda,\mu) = 0.
	\end{equation*}
	If $\mu^2 + \lambda \mu + 2\Lambda=0$, then \eqref{secondprime:eq} becomes
	\begin{equation*}
		2\mu' = -3\mu^2 - 2\Lambda.
	\end{equation*}
	If the right-hand side vanishes, then we obtain the stationary solutions
	\begin{equation}\label{eq:general-stationary}
	\mu = \pm \sqrt{-\frac{2\Lambda}{3}}, \quad \lambda = -\mu-\frac{2\Lambda}{\mu} =\pm 2\sqrt{-\frac{2\Lambda}{3}}. 
	\end{equation}
	Otherwise, we can separate the variables and integrate to obtain
	\begin{equation*}
		\log\left|\frac{\mu + \sqrt{-2\Lambda/3}}{\mu - \sqrt{-2\Lambda/3}}\right| = \sqrt{-6\Lambda}(t - t_0),
	\end{equation*}
	where $t_0$ is an integration constant. This gives the solutions
	\begin{equation*}
		\mu =  \sqrt{-\frac{2\Lambda}{3}} \frac{ e^{\sqrt{-6\Lambda}(t - t_0)} \pm 1}{ e^{\sqrt{-6\Lambda}(t - t_0)} \mp 1}, \quad \lambda = -\mu-\frac{2\Lambda}{\mu} =  2\sqrt{-\frac{2\Lambda}{3}}\frac{ e^{2\sqrt{-6\Lambda}(t - t_0)} \pm 4e^{\sqrt{-6\Lambda}(t - t_0)} + 1}{ e^{2\sqrt{-6\Lambda}(t - t_0)} - 1}.
	\end{equation*}
	Both the upper and lower solutions are well-defined eveywhere except $t=t_0$. Moreover, the vanishing sets of $\mu^2 + \lambda\mu + 2\Lambda$ and $\mu^2 + 2\lambda\mu + 4\Lambda$ do not intersect. Thus, their maximal domains of definition are the open intervals $\left]-\infty,t_0\right[$ and $\left]t_0,+\infty\right[$.

	Now we consider the case $P_\Lambda(\lambda,\mu)=0$. Observe that at $\lambda = 0$,  we have
	\begin{equation*}
		P_\Lambda(0,\mu)= \left(\mu^2 + \frac{11\Lambda}{3}\right)^2 + \frac{7\Lambda^2}{9} > 0.
	\end{equation*}
	Thus, it follows that $\lambda\neq 0$ on the vanishing set of $P_\Lambda$. We can therefore define $\sigma = \frac{\mu}{\lambda}$ and rewrite $P_\Lambda(\lambda,\mu)=0$ in terms of it as
	\begin{equation*}
		(1 + \sigma)^3\sigma +\frac{2\Lambda}{3\lambda^2}(3 +18\sigma + 11 \sigma^2) + \frac{128\Lambda^2}{9\lambda^4}=0.
	\end{equation*} 
	This is quadratic in $\frac{1}{\lambda^2}$, so we can solve for it to obtain
	\begin{equation}\label{eq:lamsig}
		\frac{1}{\lambda^2} = \frac{-3(3 + 18\sigma +11\sigma^2)\pm 3\sqrt{(1-\sigma)^3(9+7\sigma)}}{128\Lambda}.
	\end{equation}
	In order for the right-hand side to be real, we must have $-\frac{9}{7}\le \sigma \le 1$. Given that $\Lambda<0$, the upper solution is positive for $0<\sigma \le 1$ and the lower solution is positive for $-1<\sigma \le 1$. Note however that the case $\sigma = 1$ has to be excluded as it implies $\mu = \lambda = \pm 2\sqrt{-\Lambda/3}$ and these are precisely the points where the vanishing sets of $P_\Lambda$ and $\mu^2 + 2\lambda\mu + 4\Lambda$ intersect. To summarise, the allowed range for $\sigma$ in the upper solution is $\left]0,1\right[$ while that in the lower solution is $\left]-1,1\right[$. Moreover, for each such case, there are two solutions for $\lambda$, one positive and one negative. 
	
	Meanwhile, from \eqref{firstprime:eq} and \eqref{secondprime:eq}, we can derive the following ode for $\sigma$:
	\begin{equation*}
		2\sigma' = 2\sigma(1+\sigma )(2+\sigma)\lambda + \frac{4\Lambda(1+3\sigma )}{\lambda}.
	\end{equation*}
	We can now obtain separable odes for $\sigma$ by substituting the solutions for $\lambda$ in \eqref{eq:lamsig} into the above:
	\begin{align}
		&\mp \frac{8\sqrt{3 +18\sigma + 11\sigma^2 -\sqrt{(1-\sigma)^3(9+7\sigma)}}}{\sqrt{-2\Lambda/3}}\,\sigma'\nonumber\\
		&\qquad\qquad=(1-\sigma)(1-5\sigma)(9+7\sigma) - 3(1+3\sigma)\sqrt{(1-\sigma)^3(9+7\sigma)},\label{eq:sigmaode1}\\
		&\mp \frac{8\sqrt{3 +18\sigma + 11\sigma^2 +\sqrt{(1-\sigma)^3(9+7\sigma)}}}{\sqrt{-2\Lambda/3}}\,\sigma'\nonumber\\
		&\qquad\qquad=(1-\sigma)(1-5\sigma)(9+7\sigma) + 3(1+3\sigma)\sqrt{(1-\sigma)^3(9+7\sigma)}.\label{eq:sigmaode2}
	\end{align}
	The $\mp$ in the above refers to the choice of sign $\pm$ of $\lambda$. The allowed range for $\sigma$ in  \eqref{eq:sigmaode1} is $\left]0,1\right[$ while that in \eqref{eq:sigmaode2} is $\left]-1,1\right[$. The right-hand side in \eqref{eq:sigmaode1} is non-vanishing for all $\sigma$ in the allowed range $\left]0,1\right[$, while the right-hand side in \eqref{eq:sigmaode2} vanishes only at $\sigma = \frac12$ in the allowed range $\left]-1,1\right[$. This corresponds to the stationary solution at $2\mu = \lambda =\pm 2\sqrt{-2\Lambda/3}$ that we already encountered in \eqref{eq:general-stationary}. On the complement of this, we can separate the variables and integrate to obtain
	\begin{align}
		\int_{\sigma_0}^\sigma \frac{8\sqrt{3 +18u + 11u^2 -\sqrt{(1-u)^3(9+7u)}}\,du}{(1-u)(1-5u)(9+7u) - 3(1+3u)\sqrt{(1-u)^3(9+7u)}}=\mp \sqrt{-\frac{2\Lambda}{3}}(t-t_0),\label{eq:integral1}\\
		\int_{\sigma_0}^\sigma \frac{8\sqrt{3 +18u + 11u^2 +\sqrt{(1-u)^3(9+7u)}}\,du}{(1-u)(1-5u)(9+7u) + 3(1+3u)\sqrt{(1-u)^3(9+7u)}}=\mp \sqrt{-\frac{2\Lambda}{3}}(t-t_0),\label{eq:integral2}
	\end{align}
	where $u$ is a dummy integration variable while $\sigma_0$ and $t_0$ are integration constants. The integration constants are redundant and the choice of $\sigma_0$ may be absorbed into the choice of $t_0$.

	We now make the general observation that if the integrand $f(u)$ of a given integral $F(\sigma):=\int^\sigma_{\sigma_0} f(u)du$ has the asymptotic behaviour $f(u)\sim(u-u_0)^\alpha$ as $u\rightarrow u_0$ and is well-defined over the half-closed interval $\left[\sigma_0,u_0\right[$ (if $\sigma_0 < u_0$) or $\left]u_0,\sigma_0\right]$ (if $\sigma_0 > u_0$), then $F(\sigma)$ converges in the limit $\sigma \rightarrow u_0$ when $\alpha > -1$ and diverges otherwise. For \eqref{eq:integral1},  the integrand has the asymptotic behaviour
	\begin{align*}
		\frac{8\sqrt{3 +18u + 11u^2 -\sqrt{(1-u)^3(9+7u)}}}{(1-u)(1-5u)(9+7u) - 3(1+3u)\sqrt{(1-u)^3(9+7u)}}&\sim -\frac{1}{\sqrt{3u}} \, &\mbox{as} \; u \rightarrow 0,\\
		\frac{8\sqrt{3 +18u + 11u^2 -\sqrt{(1-u)^3(9+7u)}}}{(1-u)(1-5u)(9+7u) - 3(1+3u)\sqrt{(1-u)^3(9+7u)}}&\sim  -\frac{2\sqrt{2}}{3\sqrt{(1-u)^3}}  \, &\mbox{as} \; u \rightarrow 1.
	\end{align*}
	Thus, the integral is well-defined in the limit $\sigma\rightarrow 0$, so we can set $\sigma_0=0$. Then, as $\sigma \rightarrow 0$, we have $t \rightarrow t_0$, while as $\sigma \rightarrow 1$, we have $t \rightarrow \pm \infty$. For all other values of $\sigma$ between these two limits, the integral is well-defined. So, the maximal domains of definition of the upper and lower solutions are $\left]t_0,+\infty\right[$ and $\left]-\infty,t_0\right[$ respectively. These can be combined into a single solution \eqref{eq:sol2}.

	For \eqref{eq:integral2}, the integrand has the asymptotic behaviour
	\begin{align*}
		\frac{8\sqrt{3 +18u + 11u^2 +\sqrt{(1-u)^3(9+7u)}}}{(1-u)(1-5u)(9+7u) + 3(1+3u)\sqrt{(1-u)^3(9+7u)}}&\sim -\frac{\sqrt{u+1}}{2} \, &\mbox{as}& \; u \rightarrow -1,\\
		\frac{8\sqrt{3 +18u + 11u^2 +\sqrt{(1-u)^3(9+7u)}}}{(1-u)(1-5u)(9+7u) + 3(1+3u)\sqrt{(1-u)^3(9+7u)}}&\sim- \frac{1}{\big(u-\frac12\big)} \, &\mbox{as}& \; u \rightarrow \frac{1}{2},\\
		\frac{8\sqrt{3 +18u + 11u^2 +\sqrt{(1-u)^3(9+7u)}}}{(1-u)(1-5u)(9+7u) + 3(1+3u)\sqrt{(1-u)^3(9+7u)}}&\sim \frac{2\sqrt{2}}{3\sqrt{(1-u)^3}} \, &\mbox{as}& \; u \rightarrow 1.
	\end{align*}
	The integral is ill-defined when $\sigma=\frac12$ lies in the domain of integration. So we have two qualitatively different choices, namely $\sigma_0 <\frac12$ and $\sigma_0 >\frac12$. In the first case, since the integral is well-defined in the limit $\sigma\rightarrow -1$, we can set $\sigma_0 = -1$. Then, as $\sigma \rightarrow -1$, we have $t \rightarrow t_0$, while as $\sigma \rightarrow \frac12$ from below, we have $t \rightarrow \mp \infty$. For all other values of $\sigma$ between these two limits, the integral is well-defined. So, the maximal domains of definition of the upper and lower solutions are $\left]-\infty,t_0\right[$ and $\left]t_0,+\infty\right[$ respectively. These can be combined into a single solution \eqref{eq:sol3}.
	
	However, if $\sigma_0 >\frac12$, say $\sigma=\frac34$, then as $\sigma\rightarrow \frac12$ from above, we have $t\rightarrow \pm\infty$ while as $\sigma \rightarrow 1$, we have $t\rightarrow \mp\infty$. For all other values of $\sigma$ between these limits, the integral is well-defined, so we have two solutions defined for all $t\in \mathbb{R}$, namely \eqref{eq:sol4}.
	
	It remains to check which of the solutions satisfy the sign constraint $\mu^2 +2\lambda\mu +4\Lambda<0$ necessary for the metric to be positive definite. Note that the condition $\mu^2 +\lambda\mu +2\Lambda=0$ automatically implies
		\begin{equation*}
			\mu^2 +2\lambda\mu +4\Lambda = 2(\mu^2 +\lambda\mu +2\Lambda) - \mu^2 = -\mu^2 < 0.
		\end{equation*}
	So the stationary solutions and the solution \eqref{eq:sol1} satisfy the sign constraint.
	
	For the rest of the solutions, we see using \eqref{eq:lamsig} that
	\begin{equation*}
		\sigma^2 + 2\sigma + \frac{4\Lambda}{\lambda^2} =   \frac{(\sigma-9)(1-\sigma) \pm 3\sqrt{(1-\sigma)^3(9+7\sigma)}}{32}.
	\end{equation*}
	For $-1<\sigma < 1$, the right-hand side is positive for the upper solution and negative for the lower solution. Multiplying by $\lambda^2$ throughout then tells us that \eqref{eq:sol2} (which corresponds to the upper solution) does not satisfy the sign constraint, while \eqref{eq:sol3} and \eqref{eq:sol4} (which correspond to the lower solution) do satisfy the sign constraint.
\end{proof}

\begin{rem}\label{rem:UHasKzero}
	For $\Lambda=-6$, the two stationary solutions are the solutions associated to the complex hyperbolic plane with holomorphic sectional curvature $-4$, while the non-stationary solutions  in \eqref{eq:sol3} and \eqref{eq:sol4}  are the solutions associated to the one-loop deformed universal hypermultiplet. 
	
	The change of coordinates between $\sigma$ and $\rho$ may be deduced from \eqref{eq:D} to be
	\begin{equation}\label{eq:rhobyc}
		\frac{\rho}{c} = -8\left(5 \mp  \sqrt{\frac{9+7\sigma}{1-\sigma}}\right)^{-1}. 
	\end{equation}
	This is well-defined for $-1 < \sigma < 1$.
	 Multiplying \eqref{mu-rho:eq} by $\frac{dt}{d\sigma}$ and using the chain rule on the right-hand side gives us
	 \begin{equation}\label{eq:consistency}
	 	\mu \frac{dt}{d\sigma} = -\frac{1}{\rho}\left(\frac{\rho +4c}{\rho + 2c}\right)\frac{d\rho}{d\sigma}.
	 \end{equation}
 	Using the explicit expressions in \eqref{eq:sol3} and \eqref{eq:sol4} we find that
 	\begin{equation*}
 		\mu \frac{dt}{d\sigma} = \frac{-64\sigma}{(1-\sigma)(1-5\sigma)(9 + 7\sigma)+3(1+3\sigma)\sqrt{(1-\sigma)^3(9+7\sigma)}}.
 	\end{equation*}
 	Meanwhile, using \eqref{eq:rhobyc} we find that 
 	\begin{equation*}
 		-\frac{1}{\rho}\left(\frac{\rho +4c}{\rho + 2c}\right)\frac{d\rho}{d\sigma} = \frac{-64\sigma}{(1-\sigma)(1-5\sigma)(9 + 7\sigma)\pm 3(1+3\sigma)\sqrt{(1-\sigma)^3(9+7\sigma)}}.
 	\end{equation*}
	Thus, we see that \eqref{eq:consistency} holds only for the upper solution 
	\begin{equation}
			\frac{\rho}{c} = -8\left(5 - \sqrt{\frac{9+7\sigma}{1-\sigma}}\right)^{-1} = -\frac{1-\sigma}{2(1 -2\sigma)}\left(5+\sqrt{\frac{9 + 7\sigma}{1 - \sigma}}\right).
	\end{equation}
	When $\frac{1}{2}<\sigma <1$, we have $\frac{\rho}{c}>0$, while when $-1<\sigma< \frac{1}{2}$, we have $\frac{\rho}{c}<-2$. This is consistent with the fact that the one-loop deformed universal hypermultiplet metric is complete over the domain $\frac{\rho}{c}>0$ and incomplete over the domain $\frac{\rho}{c}<-2$.

\end{rem}

In principle, for arbitrary values of $K$, the constraint \eqref{eq:flowlines} allows us to write $\lambda$ as an implicit function of $\mu$, which can then be used to turn \eqref{secondprime:eq} into a separable ode in $\mu$. However, as \eqref{eq:flowlines} amounts to a bivariate polynomial equation of degree $7$, the implicit function cannot be expected to have a closed form in terms of radicals. Nevertheless, it is possible to make conclusions about the completeness of the solutions, as in the next theorem. 
From Section \ref{Ricci-flat:sec} and Myer's theorem we know that the Einstein constant $\Lambda$ of any complete solution of  the Einstein equations \eqref{firstprime:eq}-\eqref{thirdprime:eq} is necessarily negative. So we may as well assume $\Lambda = -6$.

\begin{thm}\label{main:thm}
	The complex hyperbolic metric \eqref{cxhyp:eq} of constant holomorphic sectional curvature $-4$ and the one-loop deformed universal hypermultiplet metric  with $\frac{\rho}{c}>0$ are the only complete solutions of the Einstein equations \eqref{firstprime:eq}-\eqref{thirdprime:eq} for $\Lambda = -6$.
\end{thm}
\begin{proof}
	Suppose we have a complete solution $(\lambda(t), \mu(t))$ of \eqref{firstprime:eq}-\eqref{thirdprime:eq}. The solution is either bounded in both the limits $t\rightarrow \pm\infty$ or unbounded in at least one.

	We first consider the bounded case. Fix a real number $k$ and define
	\begin{align*}
		f(\lambda,\mu)&= 2(\lambda + k\mu)P_\Lambda(\lambda,\mu)^2, \\
		h(\lambda,\mu)&= \lambda^2+(2+12k)\mu^2+(18-3k)\lambda\mu +(12-4k) \Lambda.
	\end{align*}
	By \eqref{firstprime:eq}, \eqref{secondprime:eq}, and \eqref{eq:Plamb}, we have
	\begin{equation*}
		f' = -hP_\Lambda^2.
	\end{equation*}
	The function $h$ is strictly positive, and so $-hP_\Lambda^2$ is non-positive, for all $\lambda,\mu$ and all $\Lambda<0$ whenever $k$ satisfies
	\begin{equation}\label{eq:range}
		3 < k \le \frac{26 + 6\sqrt{10}}{3}.
	\end{equation}
	Thus, given that $k$ is in the above range, we have a  monotonically decreasing function $f(\lambda(t),\mu(t))$ of $t$. 
	
	Any monotonically decreasing function of $t$ is either unbounded or has well-defined (finite) limits as $t\rightarrow \pm \infty$. Since our solution is assumed to be bounded, it has to be the latter case. In fact, 
	the limiting value of $f$ must be one for which $f'=-hP_\Lambda^2$ (and hence $P_\Lambda$) vanishes. Thus, we have a well-defined limit
	\begin{equation*}
		\lim_{t\rightarrow \pm\infty} P_\Lambda(\lambda(t),\mu(t)) = 0.
	\end{equation*}	
	The constant $K$ in the constraint \eqref{eq:flowlines} is either zero or non-zero. We have already explicitly described the $K=0$ case in Proposition \ref{prop:K-constraint}  and seen in Remark \ref{rem:UHasKzero} that the only complete solutions for $\Lambda = -6$ correspond to precisely the complex hyperbolic metric with holomorphic sectional curvature $-4$ and the one-loop deformed universal hypermultiplet metric with $\frac{\rho}{c}>0$. So we may assume $K \neq 0$ now. Since the solution is bounded, this implies that
	\begin{equation*}
		\lim_{t\rightarrow \pm\infty} (\mu(t)^2 + 2\lambda(t)\mu(t) + 4\Lambda)^3 = 	\lim_{t\rightarrow \pm\infty} \frac{1}{K}(\mu(t)^2 + \lambda(t)\mu(t) + 2\Lambda)^3P_\Lambda(\lambda(t),\mu(t)) = 0.
	\end{equation*}
	The vanishing sets of the polynomials $\mu^2 + 2\lambda \mu+ 4\Lambda$ and $P_\Lambda$ intersect in precisely the 
	points $\pm \Big(  \frac{2\sqrt{-\Lambda}}{\sqrt 3},  \frac{2\sqrt{-\Lambda}}{\sqrt 3}\Big)$. 
	 These can be checked to be fixed points of the first-order ode system \eqref{firstprime:eq} and \eqref{secondprime:eq}.  Moreover, $\Big(  \frac{2\sqrt{-\Lambda}}{\sqrt 3},  \frac{2\sqrt{-\Lambda}}{\sqrt 3}\Big)$ is a stable fixed point, while $-\Big(  \frac{2\sqrt{-\Lambda}}{\sqrt 3},  \frac{2\sqrt{-\Lambda}}{\sqrt 3}\Big)$ is an unstable fixed point. It thus follows that
	\begin{equation*}
	\lim_{t\rightarrow \pm\infty} (\lambda(t),\mu(t)) = \pm\bigg( \frac{2\sqrt{-\Lambda}}{\sqrt 3},  \frac{2\sqrt{-\Lambda}}{\sqrt 3}\bigg).
	\end{equation*}
	Now, as the vanishing set of $\mu^2 + \lambda \mu + 2\Lambda$ is a hyperbola, its complement in $\R^2$ consists of three connected components. The two fixed points 
	$\pm\Big(  \frac{2\sqrt{-\Lambda}}{\sqrt 3},  \frac{2\sqrt{-\Lambda}}{\sqrt 3}\Big)$ are on two different connected components, hence any complete solution has to intersect the hyperbola $\mu^2 + \lambda \mu + 2\Lambda=0$. But this is not possible since $K$ vanishes if  $\mu^2 + \lambda \mu + 2\Lambda$ vanishes, and we have assumed that $K\neq 0$.
	
	Next, we consider the case where the solution $(\lambda(t), \mu(t))$ is unbounded in at least one of the limits $t\rightarrow \pm\infty$. Without loss of generality, we may assume the solution is unbounded as $t\rightarrow +\infty$. Then there exists a sequence $t_i\rightarrow +\infty$ such that at 
	least one of the following holds: 
	\begin{equation*}
		\lim_{i\rightarrow \infty}   \frac{1}{\lambda(t_i)}=	\lim_{i\rightarrow \infty}  \left(\frac{1}{\lambda(t_i)}\right)'=0 \quad \mbox{or} \quad \lim_{i\rightarrow \infty}  \frac{1}{\mu(t_i)}=\lim_{i\rightarrow \infty}  \left(\frac{1}{\mu(t_i)}\right)'=0.
	\end{equation*}
	
	Since the solution satisfies a polynomial constraint, by regarding it  as the vanishing set of a polynomial on $\R\mathbb{P}^2$, we  even see that 
	$\lim_{t\rightarrow \infty} \frac{1}{\lambda (t)} =0$ or  $\lim_{t\rightarrow \infty} \frac{1}{\mu (t)} =0$.

	If we let $\sigma = \frac{\mu}{\lambda}$ and $\nu = \frac{\lambda}{\mu}$, then the constraint \eqref{eq:flowlines} may be rewritten in the following manner over suitable domains:
	\begin{align*}
	\left(\frac{\sigma +1 +2\Lambda\lambda^{-2}}{\sigma +2 +4\Lambda\lambda^{-2}}\right)^3 \left((1 + \sigma)^3\sigma +\frac{2\Lambda}{3\lambda^2}(3 +18\sigma + 11\sigma^2) + \frac{128\Lambda^2}{9\lambda^4}\right) &= \frac{K}{\lambda^4},\\
	\left(\frac{1 +\nu +2\Lambda\mu^{-2}}{1 +2\nu +4\Lambda\mu^{-2}}\right)^3 \left((\nu + 1)^3 +\frac{2\Lambda}{3\mu^2}(3\nu^2 +18\nu + 11) + \frac{128\Lambda^2}{9\mu^4}\right) &= \frac{K}{\mu^4},
	\end{align*}
	where $K$ is some constant.	The above imply that in the limit $|\lambda| \rightarrow \infty$, we have either $\sigma = -1$ or $\sigma = 0$, while in the limit $|\mu| \rightarrow \infty$, we have $\nu = -1$. Meanwhile \eqref{firstprime:eq} and \eqref{secondprime:eq} imply that
	\begin{align*}
	\left(\frac{2}{\lambda}\right)' &= 1 + 2\sigma^2 + 6\sigma + \frac{12\Lambda}{\lambda^2},\\
	\left(\frac{2}{\mu}\right)' &= -3\nu -\frac{4\Lambda}{\mu^2}.
	\end{align*}
	Thus, in the limit $|\lambda| \rightarrow \infty$, we have either $\big(\frac{2}{\lambda}\big)' \rightarrow -3$ or $\big(\frac{2}{\lambda}\big)'\rightarrow  1$, while in the limit $|\mu| \rightarrow \infty$, we have $\big(\frac{2}{\lambda}\big)'\rightarrow3$. This gives a contradiction.
\end{proof}

\end{document}